\documentclass{amsart}

\usepackage[inactive]{srcltx} 
\usepackage{amsmath, amsthm, amscd, amsfonts, amssymb, graphicx, color}
 \newtheorem{thm}{Theorem}[section]

 \theoremstyle{definition}
 
 \theoremstyle{remark}
 \newtheorem{rem}[thm]{Remark}

\begin{document}

\title[Left or right centralizers on $ \star $-algebras]{Left or right centralizers on $ \star $-algebras through orthogonal elements }

\author{ Hamid Farhadi}
\thanks{{\scriptsize
\hskip -0.4 true cm \emph{MSC(2020)}: 15A86; 47B49; 47L10; 16W10.
}}

\address{Department of Mathematics, Faculty of Science, University of Kurdistan, P.O. Box 416, Sanandaj, Kurdistan, Iran}

\email{h.farhadi@uok.ac.ir}


\begin{abstract}
In this paper we consider the problem of characterizing linear maps on special $ \star $-algebras behaving like left or right centralizers at orthogonal elements and obtain some results in this regard.
\end{abstract}

\maketitle


\section{Introduction}

Throughout this paper all algebras and vector spaces will be over the complex field $ \mathbb{C} $. Let $ \mathcal{A} $ be an algebra. Recall that a linear (additive) map $ \varphi : \mathcal{A} \to \mathcal{A} $ is said to be a \textit{right $($left$)$ centralizer} if $ \varphi (ab) = a \varphi(b)
(\varphi(ab) = \varphi(a)b) $ for each $a, b \in \mathcal{A}$. The map $ \varphi $ is called a \textit{centralizer} if it is both a left centralizer and a right centralizer. In the case that $ \mathcal{A} $ has a unity $1$, $ \varphi : \mathcal{A} \to \mathcal{A} $ is a right (left) centralizer if and only if $ \varphi $ is of the form $ \varphi (a) = a \varphi(1) ( \varphi(a) = \varphi(1)a)$ for all $a \in \mathcal{A}$. Also $ \varphi $ is a centralizer if and only if $ \varphi (a) = a \varphi(1) = \varphi(1)a$ for each $a \in \mathcal{A}$. The notion of centralizer appears naturally in $C^{*}$-algebras. In ring theory it is more common to work with module homomorphisms. We refer the reader to \cite{gh1, gh2, vuk} and references therein for results concerning centralizers on rings and algebras.

In recent years, several authors studied the linear (additive) maps that behave like homomorphisms, derivations or right (left) centalizers when acting on special products (for instance, see \cite{barar, bre, fad0, fad1, fad2} and the references therein). An algebra $ \mathcal{A} $ is called \textit{zero product determined} if for every linear space $\mathcal{X}$ and every bilinear map $\phi:\mathcal{A}\times \mathcal{A}\rightarrow \mathcal{X}$ the following holds: If $\phi(a,b)=0$ whenever $ab=0$, then there exists a linear map $T : \mathcal{A}^{2}\rightarrow \mathcal{X}$ such that $\phi(a,b)=T(ab)$ for each $a,b\in \mathcal{A}$. If $\mathcal{A}$ has unity $1$, then $\mathcal{A}$ is zero product determined if and only if for every linear space $\mathcal{X}$ and every bilinear map $\phi:\mathcal{A}\times \mathcal{A}\rightarrow \mathcal{X}$, the following holds: If $\phi(a,b)=0$ whenever $ab=0$, then $\phi(a,b)=\phi(ab,1)$ for each $a,b\in \mathcal{A}$. Also in this case $\phi(a,1)=\phi(1,a)$ for all $a\in \mathcal{A}$. The question of characterizing linear maps through zero products, Jordan products, etc. on algebras sometimes can be effectively solved by considering bilinear maps that preserve certain zero product properties (for instance, see \cite{al, al1, fos, gh3,gh4,gh5,gh6, gh7, gh8}). Motivated by these works, Bre\v{s}ar et al. \cite{bre2} introduced the concept of zero product (Jordan product) determined algebras, which can be used to study linear maps preserving zero products (Jordan products) and derivable (Jordan derivable) maps at zero point.  
\par
Let $ \varphi : \mathcal{A} \to   \mathcal{A} $ be a linear mapping on algebra $  \mathcal{A} $. A tempting
challenge for researchers is to determine conditions on a certain set $ \mathcal{S} \subseteq  \mathcal{A} \times \mathcal{A} $ to guarantee that the property
\begin{equation} \label{1}
\varphi (ab) = a \varphi(b)\quad   \big (\varphi(ab) = \varphi(a)b\big), \text{ for every } (a, b) \in  \mathcal{S} , 
\end{equation} 
implies that $ \varphi $ is a (right, left) centralizer. Some particular subsets $ \mathcal{S} $ give rise to precise
notions studied in the literature. For example, given a fixed element $z \in \mathcal{A}$, a
linear map $ \varphi : \mathcal{A} \to   \mathcal{A} $ satisfying \eqref{1} for the set $\mathcal{S}_{z} = \{ (a, b) \in  \mathcal{A} \times \mathcal{A} : ab = z \} $ is called \textit{centralizer} at $z$. Motivated by \cite{barar, fad1, fad2, gh7, gh8} in this paper we consider the problem of characterizing linear maps on special $ \star $-algebras behaving like left or right centralizers at orthogonal elements for several types of orthogonality conditions.
\par 
In this paper we consider the problem of characterizing linear maps behaving like right or left centralizers at orthogonal elements for several types of orthogonality conditions on $ \star $-algebras with unity. In particular, in this paper we consider the subsequent conditions on a linear map $ \varphi :  \mathcal{A} \to \mathcal{A}  $ where $ \mathcal{A} $ is a zero product determined $ \star $-algebra with unity or $ \mathcal{A} $ is a unital standard operator algebras on a Hilbert space $H$ such that $ \mathcal{A} $ is closed under adjoint operation : 
\[   a, b \in \mathcal{A} , a b^\star =0 \Longrightarrow a \varphi(b)^\star = 0  ; \]
\[  a, b \in \mathcal{A} , a^\star b =0 \Longrightarrow   \varphi(a)^\star b = 0. \]
Let $H$ be a Hilbert space. We denote by $B(H)$ the algebra of all bounded linear operators on $H$, and $F(H)$ denotes the algebra of all finite rank operators in $B(H)$. Recall that a \textit{standard operator algebra} is any subalgebra of $B(H)$ which contains $F(H)$. We shall denote the identity matrix of $B(H)$ by $I$. 


\section{Main results}
We first characterize the centralizers at orthogonal elements on unital zero product determined $ \star $-algebras.

\begin{thm} \label{tc}
Let $ A $ be a zero product determined $ \star $-algebra with unity $1$ and $ \varphi : A \to A $ be a linear map. Then the following conditions are equivalent:
\begin{enumerate}
\item[(i)]
$\varphi$ is a left centralizer;
\item[(ii)]
$  a, b \in \mathcal{A} , a b^\star =0 \Longrightarrow a \varphi(b)^\star = 0  $.
\end{enumerate}
\end{thm}
\begin{proof}
$ (i) \Rightarrow (ii) $ Since $\mathcal{A}$ is unital, it follows that $\varphi(a) = \varphi(1)a$ for each $a\in \mathcal{A}$. If $ a b^\star =0$, then 
\[a \varphi(b)^\star=a(\varphi(1)b)^{\star}=a b^\star\varphi(1)^{\star} =0. \]
So (ii) holds. \\
$ (ii) \Rightarrow (i) $ Define $\phi:\mathcal{A}\times \mathcal{A}\rightarrow \mathcal{A}$ by $\phi(a,b)=a\varphi(b^{\star})^{\star}$. It is easily checked that $\phi$ is a bilinear map. If $a,b\in \mathcal{A}$ such that $ab=0$, then $a(b^{\star})^{\star}=0$. It follows from hypothesis that $a\varphi(b^{\star})^{\star}=0$. Hence $\phi(a,b)=0$. Since $\mathcal{A}$ is a zero product determined algebra, it follows that $\phi(a,b)=\phi(ab,1)$ for each $a,b\in \mathcal{A}$. Now we have 
\[ a\varphi(b^{\star})^{\star}=ab\varphi(1)^{\star}\] 
for each $a,b\in \mathcal{A}$. By letting $a=1$ we get
\[\varphi(b^{\star})^{\star}=b\varphi(1)^{\star} \] 
for each $b\in \mathcal{A}$. Thus $\varphi(b^{\star})=\varphi(1)b^{\star}$ for all $b\in \mathcal{A}$ and hence $\varphi(a)=\varphi(1)a$ for all $a\in \mathcal{A}$. Hence $\varphi$ is a left centralizer.
\end{proof}
\begin{thm} \label{tc2}
Let $ A $ be a zero product determined $ \star $-algebra with unity $1$ and $ \varphi : A \to A $ be a linear map. Then the following conditions are equivalent:
\begin{enumerate}
\item[(i)]
$\varphi$ is a right centralizer;
\item[(ii)]
$  a, b \in \mathcal{A} , a^\star b =0 \Longrightarrow   \varphi(a)^\star b = 0 $.
\end{enumerate}
\end{thm}
\begin{proof}
$ (i) \Rightarrow (ii) $ Since $\mathcal{A}$ is unital, it follows that $\varphi(a) = a\varphi(1)$ for each $a\in \mathcal{A}$. If $ a^\star b=0$, then 
\[ \varphi(a)^\star b=(a\varphi(1))^{\star}= \varphi(1)^{\star}a^\star b=0. \]
So (ii) holds. \\
$ (ii) \Rightarrow (i) $ Define the bilinear map $\phi:\mathcal{A}\times \mathcal{A}\rightarrow \mathcal{A}$ by $\phi(a,b)=\varphi(a^{\star})^{\star}b$. If $a,b\in \mathcal{A}$ such that $ab=0$, then $(a^{\star})^{\star}b=0$. By hypothesis $\varphi(a^{\star})^{\star}b=0$. So $\phi(a,b)=0$. Since $\mathcal{A}$ is a zero product determined algebra, it follows that $\phi(a,b)=\phi(ab,1)=\phi(1,ab)$ for each $a,b\in \mathcal{A}$. Now 
\[ \varphi(a^{\star})^{\star}b=\varphi(1)^{\star}ab\] 
for each $a,b\in \mathcal{A}$. By letting $b=1$ we arrive at
\[\varphi(a^{\star})^{\star}=\varphi(1)^{\star}a \] 
for each $a\in \mathcal{A}$. Thus $\varphi(a^{\star})=a^{\star}\varphi(1)$ for all $a\in \mathcal{A}$ and hence $\varphi(a)=a\varphi(1)$ for all $a\in \mathcal{A}$, giving us $\varphi$ is a right centralizer.
\end{proof}
\begin{rem}
Every algebra which is generated by its idempotents is zero product determined \cite{bre3}. So the following algebras are zero product determined: 
\begin{itemize}
\item[(i)] Any algebra which is linearly spanned by its idempotents.
By \cite[Lemma 3. 2]{hou} and \cite[Theorem 1]{pe}, $B(H)$ is linearly spanned by its idempotents. By \cite[Theorem 4]{pe}, every element in a properly infinite $W^*$-algebra $\mathcal{A}$ is a sum of at most five idempotents. In \cite{mar} several classes of simple $C^*$-algebras are given which are linearly spanned by their projections. 
\item[(ii)] Any simple unital algebra containing a non-trivial idempotent, since these algebras are generated by their idempotents \cite{bre}. 
\end{itemize}
Therefore Theorems \ref{tc} and \ref{tc2} hold for $\star$-algebras that satisfy one of the above conditions.
\end{rem}
In the following, we will characterize the the centralizers at orthogonal elements on the unital standard operator
algebras on Hilbert spaces that are closed under adjoint operation.
\begin{thm}\label{s1}
Let $\mathcal{A}$ be a unital standard operator algebra on a Hilbert space $H$ with $dimH \geq 2$, such that $\mathcal{A}$ is closed under adjoint operation. Suppose that $ \varphi : A \to A $ is a linear map. Then the following conditions are equivalent:
\begin{enumerate}
\item[(i)]
$\varphi$ is a left centralizer;
\item[(ii)]
$  A,B \in \mathcal{A} , AB^\star =0 \Longrightarrow A \varphi(B)^\star = 0  $.
\end{enumerate}
\end{thm}
\begin{proof}
$ (i) \Rightarrow (ii) $ is similar to proof of Theorem \ref{tc}.\\
$ (ii) \Rightarrow (i) $ Define $\psi :\mathcal{A} \rightarrow \mathcal{A}$ by $\psi(A)=\varphi(A^{\star})^{\star}$. Then $\psi$ is a linear map such that 
\[ A,B \in \mathcal{A} , AB=0 \Longrightarrow A \psi(B) = 0. \]
Let $P\in \mathcal{A}$ be an idempotent operator of rank one and $P\in \mathcal{A}$. Then $P (I-P)A = 0$ and $(I-P)P A= 0$, and by assumption, we have 
\[P\psi(A)=P\psi(PA) \quad \text{and} \quad \psi(PA)=P\psi(PA) \]
So $\psi(PA)=P\psi(A)$ for all $A\in \mathcal{A}$. By \cite[Lemma 1.1]{bur}, every element $X \in F(H )$ is a linear combination of rank-one idempotents, and so 
\begin{equation}\label{e1}
\psi(XA)=X\psi(A)
\end{equation}
for all $X \in F(H )$ and $A\in \mathcal{A}$. By letting $A=I$ in \eqref{e1} we get $\psi(X)=X\psi(I)$ for all $X \in F(H )$. Since $F(H)$ is an ideal in $\mathcal{A}$, it follows that 
\begin{equation}\label{e2}
\psi(XA)=XA\psi(I)
\end{equation}
for all $X \in F(H )$. By comparing \eqref{e1} and \eqref{e2}, we see that $X\psi(A)=XA\psi(I)$ for all $X \in F(H )$ and $A\in \mathcal{A}$. Since $F(H)$ is an essential ideal in $B(H)$, it follows that $\psi(A)=A\psi(I)$ for all $A\in \mathcal{A}$. From definition of $\psi$ we have $\varphi(A^{\star})^{\star}=A\varphi(I)^{\star}$ for all $A\in \mathcal{A}$. Thus $\varphi(A^{\star})=\varphi(I)A^{\star}$ for all $A\in \mathcal{A}$ and hence $\varphi(A)=\varphi(I)A$ for all $A\in \mathcal{A}$. Thus $\varphi$ is a left centralizer.
\end{proof}
\begin{thm}\label{s2}
Let $\mathcal{A}$ be a unital standard operator algebra on a Hilbert space $H$ with $dimH \geq 2$, such that $\mathcal{A}$ is closed under adjoint operation. Suppose that $ \varphi : A \to A $ is a linear map. Then the following conditions are equivalent:
\begin{enumerate}
\item[(i)]
$\varphi$ is a right centralizer;
\item[(ii)]
$  A,B \in \mathcal{A} , A^\star B =0 \Longrightarrow   \varphi(A)^\star B = 0 $.
\end{enumerate}
\end{thm}
\begin{proof}
$ (i) \Rightarrow (ii) $ is similar to proof of Theorem \ref{tc2}.\\
$ (ii) \Rightarrow (i) $ Define $\psi :\mathcal{A} \rightarrow \mathcal{A}$ by $\psi(A)=\varphi(A^{\star})^{\star}$. Then $\psi$ is a linear map such that 
\[ A,B \in \mathcal{A} , AB=0 \Longrightarrow  \psi(A)B = 0. \]
Let $P\in \mathcal{A}$ be an idempotent operator of rank one and $P\in \mathcal{A}$. Then $AP (I-P) = 0$ and $A(I-P)P = 0$, and by assumption, we arrive at $\psi(AP)=\psi(A)P$ for all $A\in \mathcal{A}$. So 
\begin{equation}\label{e3}
\psi(AX)=\psi(A)X
\end{equation}
for all $X \in F(H )$ and $A\in \mathcal{A}$. By letting $A=I$ in \eqref{e3} we have $\psi(X)=\psi(I)X$ for all $X \in F(H )$. Since $F(H)$ is an ideal in $\mathcal{A}$, it follows that 
\begin{equation}\label{e4}
\psi(AX)=\psi(I)AX
\end{equation}
for all $X \in F(H )$. By comparing \eqref{e3} and \eqref{e4}, we get $\psi(A)X=\psi(I)AX$ for all $X \in F(H )$ and $A\in \mathcal{A}$. Since $F(H)$ is an essential ideal in $B(H)$, it follows that $\psi(A)=\psi(I)A$ for all $A\in \mathcal{A}$. From definition of $\psi$ we have $\varphi(A^{\star})^{\star}=\varphi(I)^{\star}A$ for all $A\in \mathcal{A}$. Thus $\varphi(A^{\star})=A^{\star}\varphi(I)$ for all $A\in \mathcal{A}$ and hence $\varphi(A)=A\varphi(I)$ for all $A\in \mathcal{A}$ implying that  $\varphi$ is a right centralizer.
\end{proof}
Finally, we note that the characterization of left or right centralizers through orthogonal elements can be used to study local left or right centralizers.
\bibliographystyle{amsplain}
\bibliography{xbib}

\end{document}